\documentclass{amsart}
\title{Compactness versus hugeness at successor cardinals}
\author{Sean Cox}
\author{Monroe Eskew}

\date{}

\usepackage{color} %used for font color
\usepackage{amssymb} %maths
\usepackage{amsmath} %maths
\usepackage[utf8]{inputenc} %useful to type directly accentuated characters
\usepackage{enumerate}
\usepackage{amsthm}
\usepackage{cite}
\usepackage{comment}

\newtheorem{theorem}{Theorem}[section]
\newtheorem{lemma}[theorem]{Lemma}
\newtheorem{proposition}[theorem]{Proposition}
\newtheorem{corollary}[theorem]{Corollary}

\newtheorem{claim}[theorem]{Claim}

\newtheorem{remark}[theorem]{Remark}

\DeclareMathOperator{\dom}{dom}

\DeclareMathOperator{\ot}{ot}

\DeclareMathOperator{\cf}{cf}
\DeclareMathOperator{\cof}{cof}

\DeclareMathOperator{\add}{Add}
\DeclareMathOperator{\col}{Col}
\DeclareMathOperator{\ord}{Ord}

\DeclareMathOperator{\crit}{crit}
\DeclareMathOperator{\id}{id}
\DeclareMathOperator{\ns}{NS}

\newcommand{\chang}{\twoheadrightarrow}

\newcommand{\p}{\mathcal{P}}
\newcommand{\la}{\langle}
\newcommand{\ra}{\rangle}

\newcommand{\los}{\L o\'s}

\begin{document}
\maketitle

\begin{abstract}
If $\kappa$ is regular and $2^{<\kappa}\leq\kappa^+$, then the existence of a weakly presaturated ideal on $\kappa^+$ implies $\square^*_\kappa$.  This partially answers a question of Foreman and Magidor about the approachability ideal on $\omega_2$.  As a corollary, we show that if there is a presaturated ideal $I$ on $\omega_2$ such that $\p(\omega_2)/I$ is semiproper, then CH holds.  We also show some barriers to getting the tree property and a saturated ideal simultaneously on a successor cardinal from conventional forcing methods.
\end{abstract}

The motivating question for this work is:  To what extent are large cardinal properties of small cardinals mutually consistent?  We focus here on a tension between versions of compactness and hugeness that make sense for successor cardinals.  We show that if $\kappa$ is regular and $2^{<\kappa}\leq\kappa^+$, then we cannot have both the tree property at $\kappa^+$ and generic hugeness properties of $\kappa^+$ such as $(\kappa^{++},\kappa^+) \chang (\kappa^+,\kappa)$ or the existence of a saturated ideal on $\kappa^+$.  As a corollary, we find a tight connection between the Continuum Hypothesis and the forcing properties of the Boolean algebras associated to saturated ideals on $\omega_2$.  We do not know whether these compactness and hugeness properties of successor cardinals are consistent with each other in the absence of the cardinal arithmetic assumption.

In Section \ref{prelim}, we discuss preliminaries about ideals and trees.  In Section \ref{wssi}, we derive $\square^*_\kappa$ from several generic hugeness properties of $\kappa^+$ under cardinal arithmetic assumptions that are compatible with the tree property at $\kappa^+$.  Proposition \ref{changsquare1} and Theorem \ref{idealapprox} are due to the first author, while Theorem \ref{wps} and Corollary \ref{ch} are due to the second author.  Section \ref{wsli} presents some barriers to combining compactness and hugeness properties at successor cardinals, showing that a general template for generically lifting huge and almost-huge embeddings, while collapsing the relevant cardinals down to successors, must force the failure of the tree property at the critical point.

\section{Preliminaries}
\label{prelim}
\subsection{Ideals and generic embeddings}
Proofs of many of the facts stated in this subsection can be found in \cite{MR2768692}.
An \emph{ideal} on a set $Z$ is a collection of subsets of $Z$ closed under taking subsets and pairwise unions.  If $\kappa$ is a cardinal, we say that an ideal is $\kappa$-complete if it is closed under unions of size $<\kappa$.   If $Z \subseteq \p(X)$ and $I$ is an ideal on $Z$, then we say that $I$ is \emph{normal} if for all $x \in X$, $\{ z \in Z : x \notin z \} \in I$, and $I$ closed under \emph{diagonal unions} of the form $\nabla_{x \in X} A_x := \{ z : \exists x \in z(z \in A_x) \}$.  It is not difficult to show that if $I$ is a normal ideal on $Z \subseteq \p(\lambda)$ and $\kappa \leq \lambda$, then $I$ is $\kappa$-complete if and only if for every $\alpha<\kappa$, $\{ z \in Z : \alpha \nsubseteq z \} \in I$.  The smallest normal ideal on a set $Z$ is the \emph{nonstationary ideal}, $\ns_Z$, and its dual filter is called the \emph{club filter}, which is generated by the collection of sets $C_F \subseteq Z$, where $F : X^{<\omega} \to X$ and $C_F = \{ z \in Z : F[z^{<\omega}] \subseteq z \}$. If $I$ is an ideal on $Z$ and $A \subseteq Z$, then $I \upharpoonright A$ denotes the smallest ideal containing $I \cup \{ Z \setminus A \}$.  We say $I$ \emph{concentrates on $A$} if $Z \setminus A \in I$.  We refer to the collection $\p(Z)\setminus I$ as $I^+$ and say the members are \emph{$I$-positive}.  The $\ns$-positive sets are called \emph{stationary}.

If $I$ is an ideal on $Z$, $\p(Z)/I$ is the quotient of the Boolean algebra $\p(Z)$ by the equivalence relation $A \sim B \Leftrightarrow (A\setminus B) \cup(B\setminus A)\in I$.   If we take a generic filter $G \subseteq \p(Z)/I$, then $\bigcup G$ is a ultrafilter over $\p(Z)^V$, and we can form the generic ultrapower embedding $j_G : V \to V^Z/G$.  The critical point of $j_G$ is the least $\kappa$ such that for some $A \in \bigcup G$, $I \restriction A$ is $\kappa$-complete and $A$ is the union of $\kappa$-many sets from $I$.  If $Z \subseteq \p(X)$ and $I$ is normal, then the well-founded part of $V^Z/G$, which we will identify with its transitive collapse, has height at least $|X|^+$, and the pointwise image $j_G[X]$ is represented in $V^Z/G$ by the identity function on $Z$.  \los' Theorem is quite useful in determining properties of $j_G[X]$ via what is satisfied by $I$-almost-all $z \in Z$.

The key to showing the degree of well-foundedness of $V^Z/G$ is the notion of a \emph{canonical function}.  If $\lambda$ is a cardinal and $Z \subseteq \p_\lambda(\lambda)$, these are functions from $Z$ to $\lambda$ that are forced to represent particular ordinals below $\lambda^+$ in any generic ultrapower arising from a normal ideal on $Z$.  To define the canonical function representing $\alpha$, we choose some surjection $\sigma_\alpha : \lambda \to \alpha$.  For any two surjections $f,g : \lambda \to \alpha$, the set of $z \in Z$ such that $f[z] \not= g[z]$ is nonstationary.  It is easy to check that for any normal ideal $I$ on $Z$ and any generic $G \subseteq \p(Z)/I$, $\alpha$ is represented in the ultrapower by the function $z \mapsto \ot(\sigma_\alpha[z])$, where for a set of ordinals $s$, $\ot(s)$ denotes its order-type.

We say a normal ideal $I$ on $Z \subseteq \p(X)$ is \emph{saturated} if $\p(Z)/I$ has the $|X|^+$-chain condition.  We say $I$ is \emph{presaturated} when for every collection of maximal antichains $\la\mathcal A_x : x \in X \ra$, there is a dense set of $B$ in $\p(Z)/I$ such that for all $x\in X$, $\{ A \in \mathcal A_x : A \cap B \notin I \}$ has size $\leq |X|$.  Obviously, saturation implies presaturation.  It is well-known that if $I$ is a presaturated normal ideal on $Z \subseteq \p(\lambda)$, then the generic ultrapower $V^Z/G$ is forced to be closed under $\lambda$-sequences from $V[G]$.

An ideal is called \emph{precipitous} when the generic ultrapower is forced to be well-founded.  An ideal on a successor cardinal $\kappa$ is called \emph{strong} when it is precipitous and it is forced that $j_G(\kappa) = (\kappa^+)^V$.  It is well-known that for normal ideals on successor cardinals, presaturated implies strong.  The following further weakening is due to Woodin \cite{MR2723878}.  We say that a normal ideal $I$ is \emph{weakly presaturated} if it is forced that  $V^Z/G$ is well-founded up to $(\kappa^+)^V +1$ and $j_G(\kappa) = (\kappa^+)^V$.  The following characterization comes from translating this into a forcing relation in $V$:
\begin{proposition}
Suppose $\kappa$ is a successor cardinal and $I$ is a normal ideal on $\kappa$.  The following are equivalent:
\begin{enumerate}
\item $I$ is weakly presaturated.
\item For all $I$-positive sets $A$ and all functions $f : \kappa \to \kappa$, $f$ is bounded by a canonical function on an $I$-positive subset of $A$.
\end{enumerate}
\end{proposition}

Claverie and Schindler \cite{MR2963017} showed that a strong ideal on $\omega_1$ is equiconsistent with a Woodin cardinal, but it follows from a result of Silver and Lemma \ref{changbound} below that the consistency strength of a weakly presaturated ideal on $\omega_1$ is much lower.

The principle $(\kappa^{++},\kappa^+) \chang (\kappa^+,\kappa)$, a generalized version of Chang's Conjecture, asserts that every structure $\frak A$ in a countable language on $\kappa^{++}$ contains a substructure $\frak B$ such that $|\frak B| = \kappa^+$ and $|\frak B \cap \kappa^+ | = \kappa$.  This is equivalent to saying that the collection of $z \subseteq \kappa^{++}$ of order-type $\kappa^+$, or $[\kappa^{++}]^{\kappa^+}$, is stationary in $\p(\kappa^{++})$.  An important fact is that if $(\kappa^{++},\kappa^+)\chang(\kappa^+,\kappa)$, then the set of $z \in [\kappa^{++}]^{\kappa^+}$ such that $z \cap \kappa^+$ is an ordinal is also stationary.  If we force below this set in $\p(\kappa^{++})/\ns$, then the critical point of the embedding will be $\kappa^+$.  

We will use the following result from \cite{MR1359154}:

\begin{theorem}[Foreman-Magidor]
\label{changcof}
For any infinite cardinal $\kappa$, the set $\{ z \in [\kappa^{++}]^{\kappa^+} : \cf(z \cap \kappa^+) \not= \cf(\kappa) \}$ is nonstationary.
\end{theorem}

A connection between Chang's Conjecture and saturation properties of ideals is given by the following:

\begin{lemma}
\label{changbound}
Suppose $(\kappa^{++},\kappa^+) \chang (\kappa^+,\kappa)$.  Then there is a weakly presaturated ideal on $\kappa^+$ concentrating on $\{\alpha : \cf(\alpha) = \cf(\kappa)\}$.
\end{lemma}

\begin{proof}
Let $I$ be the nonstationary ideal on $\p(\kappa^{++})$ restricted to $Z = \{ z \subseteq \kappa^{++} : z \cap \kappa^+ \in \kappa^+ \wedge \ot(z) = \kappa^+ \}$.  Let $\pi : Z \to \kappa^+$ be defined by $\pi(z) = z \cap \kappa^+$.  Let $J$ be the collection of $A \subseteq \kappa^+$ such that $\pi^{-1}[A] \in I$.  It is easy to check that $J$ is a normal ideal on $\kappa^+$.  Note that if $A \in I^+$, then $\pi[A] \in J^+$.

Suppose $f : \kappa^+ \to \kappa^+$ and $A \in J^+$.  %Let $\bar f : Z \to \kappa^+$ be defined by $\bar f(z) = f(z\cap\kappa^+)$, and 
Let $\bar A = \pi^{-1}[A]$.  For each $z \in \bar A$, there is $\alpha_z \in z$ such that $f(z \cap \kappa^+) <\ot(z\cap \alpha_z)$.  By normality, there is $\alpha<\kappa^{++}$ and $\bar B \in I^+$ such that $\alpha_z = \alpha$ for all $z \in \bar B$.  If $\sigma_\alpha : \kappa^+ \to \alpha$ is a surjection, then for $I$-almost-all $z \in \bar B$, $\sigma_\alpha \restriction z$ is a surjection from $z \cap \kappa^+$ to $z \cap \alpha$.  We may assume this holds for all $z \in \bar B$.  

Thus for all $z \in \bar B$, $f(z\cap \kappa^+) < \ot(\sigma_\alpha[z \cap \kappa^+])$.  Let $B = \pi[\bar B]$.  Then $B$ is a $J$-positive subset of $A$, and $f$ is bounded by a canonical function on $B$.  By Theorem \ref{changcof}, $\{ \beta < \kappa^+ : \cf(\beta) \not=\cf(\kappa) \} \in J$.
\end{proof}

Silver showed that $(\omega_2,\omega_1)\chang(\omega_1,\omega)$ can be forced from an $\omega_1$-Erd\H{o}s cardinal (see \cite{MR520190}), and Donder \cite{MR645907} showed that $(\omega_2,\omega_1)\chang(\omega_1,\omega)$ implies that there is an $\omega_1$-Erd\H{o}s cardinal in an inner model.  The existence of a weakly presaturated ideal on $\kappa$ clearly implies that there is no single $f : \kappa \to \kappa$ that dominates all canonical functions modulo $\ns_{\kappa}$.  Donder and Koepke \cite{MR730856} showed that for $\kappa = \omega_1$, this statement is equiconsistent with an almost ${<}\omega_1$-Erd\H{o}s cardinal, which implies the existence of $0^\sharp$.

\subsection{Trees and weak square}
A \emph{tree} is a partial order that is well-ordered below any element.
For an infinite cardinal $\kappa$, a \emph{$\kappa$-tree} is a tree of height $\kappa$ with levels of size $<\kappa$.  We say that $\kappa$ has the \emph{tree property} if every $\kappa$-tree has a cofinal branch.  A $\kappa^+$-tree is called \emph{special} if there is a function $f : T \to \kappa$ such that $x<y$ implies $f(x) \not= f(y)$.  Clearly, special $\kappa^+$-trees cannot have cofinal branches, since a branch would witness that $\kappa^+$ is not a cardinal.  Jensen \cite{MR309729} showed the existence of a special $\kappa^+$-tree is equivalent to the weak square principle $\square^*_\kappa$, which states that there is a sequence $\la \mathcal C_\alpha : \alpha < \kappa^+ \ra$ such that:
\begin{enumerate}
\item Each $\mathcal C_\alpha$ is a nonempty set of size $\leq \kappa$ consisting of closed unbounded subsets of $\alpha$, each of order-type $\leq \kappa$.
\item If $\alpha<\kappa^+$, $D \in \mathcal C_\alpha$, and $\beta <\alpha$ is a limit point of $D$, then $D \cap \beta \in \mathcal C_\beta$.
\end{enumerate}

Mitchell \cite{MR313057} showed that having the tree property at the successor of a regular cardinal is equiconsistent with a weakly compact cardinal, and the failure of $\square^*_\kappa$ for a regular $\kappa$ is equiconsistent with a Mahlo cardinal.  Starting with regular cardinals $\mu<\kappa$ such that $2^{<\mu} = \mu$ and a weakly compact $\lambda>\kappa$, Mitchell constructed a forcing extension preserving $\mu$ and $\kappa$ and in which $\lambda = \kappa^+$, $2^\mu = \lambda$, and every $\lambda$-tree has a cofinal branch.  Starting from a Mahlo $\lambda>\kappa$, the same forcing works to produce a model in which there are no special $\kappa^+$-trees.

A strictly weaker principle than $\square^*_\kappa$ is the approachability property at $\kappa^+$.  For a regular cardinal $\kappa$, we say a set $S \subseteq \kappa$ is \emph{approachable} if there is a sequence $\la a_\alpha : \alpha < \kappa \ra$ such that for a club $C \subseteq \kappa$ and all $\alpha \in S \cap C$, there is an unbounded $A \subseteq \alpha$ of order-type $\cf(\alpha)$ such that each initial segment of $A$ is in $\{ a_\beta : \beta<\alpha \}$.  The collection of approachable subsets of $\kappa$ generates a possibly non-proper normal ideal denoted by $I[\kappa]$.  Shelah \cite{MR1318912} showed that if $\kappa$ is regular, then $\kappa^+ \cap \cof({<}\kappa) \in I[\kappa^+]$, where $\cof({<}\kappa)$ denotes the class of ordinals of cofinality $<\kappa$.

The weak square principle $\square^*_\kappa$ is absolute to any outer model with the same $\kappa^+$, including outer models in which $\kappa$ is not a cardinal.  This is because it does not matter how we bound the order-types:

\begin{lemma}
\label{otreduct}
Suppose $\xi < \kappa^+$ and $\la \mathcal C_\alpha : \alpha < \kappa^+\ra$ is a sequence such that:
\begin{enumerate}
\item Each $\mathcal C_\alpha$ is a set of $\leq \kappa$ clubs in $\alpha$, each of order-type $\leq \xi$.
\item If $\alpha<\kappa^+$, $D \in \mathcal C_\alpha$, and $\beta \in \alpha \cap\lim D$, then $D \cap \beta \in \mathcal C_\beta$.
\end{enumerate}
Then $\square^*_\kappa$ holds.
\end{lemma}

\begin{proof}
	It is easy to show by induction that for each $\delta<\kappa^+$, there is a sequence $\la E_\alpha : \alpha < \delta \ra$ such that:
	\begin{enumerate}
	\item Each $E_\alpha$ is a club in $\alpha$ of order-type $\leq \kappa$.
	\item If $\alpha<\delta$ and $\beta \in \alpha \cap\lim E_\alpha$, then $E_\beta = E_\alpha \cap \beta$.
	\end{enumerate}
	Fix such a sequence $\la E_\alpha : \alpha \leq \xi \ra$.  For each $\alpha<\kappa^+$, and $D \in \mathcal C_\alpha$ define
	$$D' = \{ \beta \in D : \ot(D \cap \beta) \in E_{\ot(D)} \}.$$
	Let	$\mathcal C'_\alpha = \{ D' : D \in \mathcal C_\alpha \}$.  Clearly, for each $\alpha<\kappa^+$, $|\mathcal C'_\alpha| \leq | \mathcal C_\alpha |\leq \kappa$ and if $D \in \mathcal C_\alpha$, then $\ot(D') \leq \ot(E_{\ot(D)}) \leq \kappa$, and $D'$ is a club in $\alpha$.
	
	To show coherence, suppose $D \in \mathcal C_\alpha$ and $\beta \in \alpha \cap \lim D'$.  Then $\beta \in \lim D$, so $D \cap \beta \in \mathcal C_\beta$.  Furthermore, $\ot(D \cap \beta) \in \lim E_{\ot(D)}$, so $E_{\ot(D \cap \beta)} = E_{\ot(D)} \cap \ot(D \cap \beta)$.  $(D \cap \beta)' \in \mathcal C'_\beta$, and $(D \cap \beta)' = \{ \gamma \in D \cap \beta : \ot(D \cap \gamma) \in E_{\ot(D)} \} = D' \cap \beta$.
\end{proof}

\section{Weak square from strong ideals}
\label{wssi}

In this section, we deduce $\square^*_\kappa$ from strong hypotheses about ideals plus cardinal arithmetic assumptions that are compatible with the tree property at $\kappa^+$.  We start with some deductions from $(\kappa^{++},\kappa^+)\chang(\kappa^+,\kappa)$ and related principles that are easier and serve to illustrate the main idea.  While these are surpassed by Theorem \ref{wps} in the case where $\kappa$ is regular, the arguments from Chang's Conjecture also work in the case of singular $\kappa$.

\begin{proposition}
\label{changsquare1}
Suppose $2^\kappa = \kappa^+$ and either $(\kappa^{++},\kappa^+) \chang (\kappa^+,\kappa)$ holds, or there is a weakly presaturated ideal on $\kappa^+$ with the property that $\p(\kappa^+)^V$ is forced to be a member of the generic ultrapower.
 %$\Vdash_{\p(\kappa^+)/I} \p(\kappa^+)^V \in V^{\kappa^+}/G$.  
 Then $\square^*_\kappa$ holds.
\end{proposition}

\begin{proof}
First assume that $(\kappa^{++},\kappa^+) \chang (\kappa^+,\kappa)$ holds.  Let $N \prec H_{\kappa^{++}}$ contain $\kappa^{++}$ and have size $\kappa^{++}$.  Let $Z = \{ z \subseteq \kappa^{++} : z \cap \kappa^+ \in \kappa^+ \wedge \ot(z) = \kappa^+ \}$.  Let $G \subseteq \p(Z)/\ns$ be generic, and let $j : V \to M$ be the ultrapower embedding.  
Since $[\id]_G = j[\kappa^{++}]$, \los' Theorem implies that $\ot(j[\kappa^{++}]) = (\kappa^{++})^V = j(\kappa^+)$.
Since $j[\kappa^{++}] \in M$ and $N$ is coded by a subset of $\kappa^{++}$, $N \in M$.

Under the second hypothesis, since in general $H_{\theta^+}$ is constructible from $\p(\theta)$, we have that $H_{\kappa^{++}}^V \in M$. Thus under both assumptions, some forcing introduces a generic elementary embedding $j : V \to M$ with critical point $\kappa^+$ such that $j(\kappa^+) = (\kappa^{++})^V$, and there is a transitive structure $N \in V \cap M$ of height $(\kappa^{++})^V$, with $\p(\kappa)^V \in N$, in which $(\kappa^+)^V$ is the largest cardinal, and such that $N \models \mathrm{ZFC} - \{ \mathrm{Powerset}\}$.
Working in $M$, we construct a $\square^*_{\kappa}$ sequence as follows.  For $\alpha<(\kappa^{++})^V$ such that $N \models \cf(\alpha) = \kappa^+$, let $D_\alpha \in N$ be a club in $\alpha$ of order-type $\kappa^+$, and let $\mathcal C_\alpha = \{ D_\alpha \}$.  For $\alpha < (\kappa^{++})^V$ such that $N \models \cf(\alpha)\leq\kappa$, let $\mathcal C_\alpha$ be the collection of all clubs in $\alpha$ of order-type $<\kappa^+$ that are members of $N$. Since $N \models 2^\kappa = \kappa^+$, and $M \models |(\kappa^+)^N| = \kappa$, we have $|\mathcal C_\alpha| \leq \kappa$ for all $\alpha$.  Since $N$ satisfies enough set theory, we have that if $D \in \mathcal C_\alpha$ and $\beta<\alpha$ is a limit point of $D$, then $D \cap \beta \in \mathcal C_\beta$.  Lemma \ref{otreduct} implies that $M \models \square^*_\kappa$. Since $\kappa<\crit(j)$, $V \models \square^*_\kappa$ by the elementarity of $j$.
\end{proof}

\begin{remark}
	It is shown in \cite{2018arXiv181211768E} that if $\cf(\kappa) = \omega$ and $(\kappa^{++},\kappa^+) \chang (\kappa^+,\kappa)$, then $\square^*_\kappa$ holds, regardless of cardinal arithmetic.
	\end{remark}

\begin{proposition}
\label{changsquare2}
Suppose $\kappa^{<\kappa} \leq \kappa^+$, $(\kappa^{++},\kappa^+) \chang (\kappa^+,\kappa)$, and the set $\{ \alpha < \kappa^+ : \exists z \in [\kappa^{++}]^{\kappa^+} (\alpha = z \cap \kappa^+ )\}$ is approachable.  Then $\square^*_\kappa$ holds.
\end{proposition}

\begin{proof}
In $V$, let $\la a_\alpha : \alpha < \kappa^+ \ra$ list all elements of $\p_\kappa(\kappa^+)$.   Let $Z = \{ z \subseteq \kappa^{++} : z \cap \kappa^+ \in \kappa^+ \wedge \ot(z) = \kappa^+ \}$, and let $j: V \to M$ be a generic ultrapower embedding obtained by forcing with $\p(Z)/\ns$.  By the approachability assumption and Theorem \ref{changcof}, for all but nonstationary-many $z \in Z$, there is an unbounded $A \subseteq z\cap \kappa^+$ of order-type $\cf(\kappa)$ such that all initial segments are in $\{ a_\beta : \beta < z\cap\kappa^+ \}$. By \los' Theorem, there is in M a club $E \subseteq (\kappa^+)^V$ of order-type $\cf(\kappa)$ such that every initial segment of $E$ is in $V$.

In $V$, choose a sequence $\la b_\alpha : \alpha < \kappa^{++} \ra$ such that each $b_\alpha$ is a club in $\alpha$ of order-type $\cf(\alpha)$.  Since $j[\kappa^{++}] \in M$, this sequence is in $M$, and the $V$-cofinalities of ordinals $<j(\kappa^+)$ are definable in $M$ in terms of this parameter.  For the same reasons, $\la \p_\kappa(\alpha)^V : \alpha < (\kappa^{++})^V \ra \in M$.  Note that each element of this sequence has size $\kappa$ in $M$.

Now we construct our $\square^*_\kappa$ sequence in $M$ as follows.  If $V \models \cf(\alpha) = \kappa^+$, then enumerate $b_\alpha$ in increasing order as $\la \gamma_\beta : \beta < \kappa^+ \ra$, and let $\mathcal C_\alpha = \{\{ \gamma_\beta : \beta \in E \}\}$.  Every initial segment of $\{ \gamma_\beta : \beta \in E \} $ is in $V$.  If $V \models \cf(\alpha) = \kappa$, let $\mathcal C_\alpha = \{ b_\alpha \}$.  If $V \models \cf(\alpha)<\kappa$, let $\mathcal C_\alpha$ be the collection of all club subsets of $\alpha$ from $V$ of order-type $<\kappa$.  For each $\alpha$, $M \models |\mathcal C_\alpha|\leq \kappa$.  If $D \in \mathcal C_\alpha$ and $\beta \in \alpha \cap \lim D$, then $D \cap \beta \in V$.  Since $\ot(D) \leq \kappa$, $V \models \cf(\beta)<\kappa$, and so $D \cap \beta \in \mathcal C_\beta$.  By elementarity, $V \models \square^*_\kappa$.
\end{proof}

In order to derive $\square^*_\kappa$ from weaker hypotheses about ideals on $\kappa^+$, we have to contend with the fact there is no guarantee that a given object of size $>\kappa^+$ from $V$ is a member of the generic ultrapower $M$.  The key to our arguments will be a refinement of the \emph{approximation property} introduced by Hamkins \cite{MR2063629}.  Suppose $A$ is a set, $X \subseteq A$, and $\mathcal F \subseteq \p(A)$.  We will say that \emph{$X$ is approximated by $\mathcal F$} when for all $a \in \mathcal F$, $a \cap X \in \mathcal F$.  The $\kappa$-approximation property can be stated in these terms as follows:  For models of set theory $M \subseteq N$ and an $M$-cardinal $\kappa$, the pair $(M,N)$ satisfies the $\kappa$-approximation property when for all ordinals $\lambda \in M$, if $X \in \p(\lambda)^N$ is approximated by $\p_\kappa(\lambda)^M$, then $X\in M$.  We say that a forcing $\mathbb P$ has the $\kappa$-approximation property when it forces that the pair $(V,V[G])$ has this property.

\begin{theorem}
\label{idealapprox}
	Suppose $\kappa$ is regular, $2^{<\kappa} \leq \kappa^+$, and $I$ is a normal ideal on $\kappa^+$ such that $\p(\kappa^+)/I$ is a forcing with the $\kappa$-approximation property.  Then $I$ is not weakly presaturated.
\end{theorem}

\begin{proof}
Recall that a weak $\kappa$-Kurepa tree is a tree of height $\kappa$, with levels of size $\leq \kappa$, and with more than $\kappa$-many cofinal branches.  If $I$ is a normal ideal on $\kappa^+$ such that $\p(\kappa^+)/I$ has the $\kappa$-approximation property, then there are no weak $\kappa$-Kurepa trees $T$, since the generic embedding would necessarily add branches to $T$, whereas any branch is approximated by $\p_\kappa(T)^V$.  Baumgartner \cite{MR401472} showed that if $\kappa$ is regular, $2^{<\kappa} \leq \kappa^+$, and there are no weak $\kappa$-Kurepa trees, then $2^\kappa=\kappa^+$.

Suppose $j : V \to M$ is a generic embedding arising from forcing with a normal ideal $I$ as above.  Since $2^{<\kappa} \leq \kappa^+$, $\p_{\kappa}(\kappa^+)^V \in M$.  By the $\kappa$-approximation property, every $X \subseteq \kappa^+$ in $M$ that is approximated by $\p_{\kappa}(\kappa^+)^V$ is in $V$.  But this means that $\p(\kappa^+)^V$ is definable in $M$ as the collection of all such $X$, since for any $X \in \p(\kappa^+)^V$, $X = j(X) \cap \kappa^+ \in M$.  Since $V \models 2^\kappa = \kappa^+$, Proposition \ref{changsquare1} implies that if $I$ is weakly presaturated, then $V \models \square^*_\kappa$. 

On the other hand, $\p(\kappa^+)/I$ having the $\kappa$-approximation property implies the tree property at $\kappa^+$, since any $\kappa^+$-tree $T$ acquires a branch in $V[G]$ by looking below a node of $j(T)$ at level $(\kappa^+)^V$.  But this branch is approximated by $\p_\kappa(T)^V$ and is thus in $V$.
\end{proof}

We note that it is consistent relative to a measurable cardinal that $2^\omega=\omega_2$ and there is a precipitous normal ideal $I$ on $\omega_2$ such that $\p(\omega_2)/I$ has the $\omega_1$-approximation property.  For example, if $\kappa$ is measurable, this is forced by (a variation of) the Mitchell forcing up to $\kappa$, countable support iteration of Sacks forcing \cite{MR556894}, and by the pure side conditions forcings of Krueger \cite{MR3647840} and Neeman \cite{MR3201836}.

\begin{theorem}\label{wps}
Suppose $\kappa$ is a regular cardinal, $2^{<\kappa} \leq \kappa^+$, and there is a weakly presaturated ideal on $\kappa^+$ concentrating on $\cof(\kappa)$.   
Then $\square^*_\kappa$ holds.
\end{theorem}

\begin{proof}
We may assume that $\kappa>\omega$, since $\square^*_\omega$ always holds.  Let $\delta = (\kappa^+)^V$.  A forcing introduces an elementary embedding $j : V \to M$ with critical point $\delta$, such that $M$ is well-founded up to $(\kappa^{++})^V
+1$, $j(\delta) = (\kappa^{++})^V$, and $M \models \cf(\delta)=\kappa$.  Since $V\models\delta^{<\kappa}=\delta$, $\p_\kappa(\delta)^V \in M$.  Define in $M$ the set $\mathcal A$ of subsets of $\delta$ that are approximated by $\p_\kappa(\delta)^V$.

Fix in $M$ a club $C^* \subseteq \delta$ of order-type $\kappa$, and let $\la \xi_\alpha : \alpha < \kappa \ra$ be its increasing enumeration.  In $V$, let $\vec \sigma= \la \sigma_\alpha : \alpha < \delta \ra$ be a sequence such that $\sigma_\alpha : \kappa \to \alpha$ is a surjection, and note that $\vec \sigma \in M$.  We can write $\delta$ as the union of a continuous increasing sequence of sets of size $<\kappa$, $\la z_\alpha : \alpha<\kappa \ra$, by putting $z_\alpha = \bigcup_{\beta<\alpha}  \sigma_{\xi_\beta}[\alpha]$.  Take $N \prec H_{j(\delta)}^M$ such that $\{ C^*,\p_\kappa(\delta)^V,\vec \sigma\} \cup \delta \subseteq N$ and $M \models |N| = \kappa$.  Let $Q = \p_\kappa(\delta) \cap N$.

Recall that a prewellordering is a transitive reflexive binary relation in which every two elements are comparable, such that the quotient by the equivalence relation, $x \sim y \Leftrightarrow x\leq y \wedge y \leq x$, is a wellorder.  An ordinal $\alpha$ has cardinality $\leq \beta$ if and only if there is a prewellordering on $\beta$ whose quotient has order-type $\alpha$.  There is a natural correspondence between surjections from sets onto ordinals and prewellorderings of those sets.  For a set of ordinals $Z$ closed under the G\"odel pairing function, a set $X \subseteq Z$ codes a relation on $Z$ via this function.  If $X \subseteq Z$ codes a prewellordering whose quotient has order-type $\alpha$, let $f_X : Z \to \alpha$ be the corresponding surjection.

\begin{claim}
\label{surjind}
Suppose $X,Y \in \mathcal A$ code prewellorderings of $\delta$ of the same length.  Then
$\{ f_X[z] : z \in Q \} = \{ f_Y[z] : z \in Q \}$.
\end{claim}

\begin{proof}[Proof of claim]
Let $r \in Q$.  We need to show that there is some $s \in Q$ such that $f_X[r] = f_Y[s]$.  There is a club $C \subseteq \kappa$ such that for all $\alpha \in C$, $f_X[z_\alpha] = f_Y[z_\alpha]$.  We may assume that for all $\alpha \in C$, $z_\alpha$ is closed under G\"odel pairing.

Let $\alpha\in C$ be such that $r \subseteq z_\alpha$.  By definition, $z_\alpha \subseteq \xi_\alpha<\delta$.  Since $|z_\alpha|<\kappa$, there is $\beta<\kappa$ such that $z = \sigma_{\xi_\alpha}[\beta] \supseteq z_\alpha$.  Since $X,Y \in \mathcal A$, $X \cap z$ and $Y \cap z$ are in $V$.  Thus $X \cap z_\alpha$ and $Y \cap z_\alpha$ are in $N$.  These sets code prewellorderings of $z_\alpha$ of order-type $\eta = \ot(f_X[z_\alpha]) =   \ot(f_Y[z_\alpha])$.  

Let $h_X : z_\alpha \to \eta$ and $h_Y : z_\alpha \to \eta$ be the corresponding surjections.  Let $r' = h_X[r]$.  Note that if $\pi : \eta \to f_X[z_\alpha]$ is the unique order-preserving map, then $\pi[r'] = f_X[r]$.  Let $s = h_Y^{-1}[r']$.  Then $s \in N$, and $h_Y[s] = r'$.  Furthermore, $\pi \circ h_Y[s] = f_Y[s] = f_X[r]$, as desired.
\end{proof}

If $f$ is a function from $Z$ to an ordinal $\alpha$ and $\beta<\alpha$, let $f \downharpoonleft \beta$ be the function $g$ such that $g(\gamma) = f(\gamma)$ when $f(\gamma)<\beta$ and $g(\gamma) = 0$ otherwise.  If $R$ is a prewellordering on a set $Z$ of order-type $\alpha$, $f_R : Z \to \alpha$ is the corresponding surjection, and $\beta<\alpha$, then let $R \downharpoonleft \beta$ denote the canonical alteration of $R$ to represent $f_R \downharpoonleft \beta$, where we make $x$ equivalent to the $R$-least element of $Z$ if $f_R(x) \geq \beta$, and leave the ordering between the elements of rank $<\beta$ the same.  

\begin{claim}
\label{surjint}
If $X \in \mathcal A$ codes a prewellordering of order-type $\alpha$ and $\beta<\alpha$, then $X \downharpoonleft \beta \in \mathcal A$.  Furthermore, if $r \in Q$, then $f_X[r] \cap \beta = f_{X \downharpoonleft\beta}[s]$ for some $s \in Q$.
\end{claim} 
\begin{proof}[Proof of claim]
Suppose $y \in \p_\kappa(\delta)^V$ and $r \in Q$.  Let $\zeta_0$ be any ordinal such that $f _X(\zeta_0)=0$.  Let $\xi < \delta$ and $\beta<\kappa$ be such that $y \cup r \cup \{\zeta_0\} \subseteq \sigma_\xi[\beta]$.  Let $z$ be the closure of $\sigma_\xi[\beta]$ under G\"odel pairing, which is in $V$.  Let $h : z \to \eta$ be the surjection coded by $X \cap z$.  There is some $\xi \leq \eta$ such that $h(\gamma) <\xi \Leftrightarrow f_X(\gamma) < \beta$ for $\gamma \in z$.  The function $h \downharpoonleft \xi$ and its code $x \subseteq z$ are in $V$.  We have that $(X \downharpoonleft \beta) \cap z = x$.  Thus $(X \downharpoonleft \beta) \cap y = (X \downharpoonleft \beta) \cap z \cap y = x \cap y \in V$.  This shows that $X \downharpoonleft \beta \in \mathcal A$.

For the second part, let $s = \{\gamma \in r : h(\gamma) < \xi \}$.  Then $s \in N$, and $f_X[s] = f_{X \downharpoonleft \beta}[s] = f_X[r] \cap \beta$.
\end{proof}

To define a $\square^*_\kappa$-sequence in $M$, first consider ordinals $\alpha<j(\delta)$ of cofinality $<\kappa$.  Let $\mathcal C_\alpha$ be the set of all clubs $D$ in $\alpha$ of order-type $<\kappa$, such that for some $X \in \mathcal A$ that codes a prewellordering of $\delta$ of order-type $\alpha$, $D = f_X[s]$ for some $s \in Q$.  By Claim \ref{surjind}, the choice of $X$ does not matter, so the cardinality of this set is at most $|N| =\kappa$.  By Claim \ref{surjint}, if $C \in \mathcal C_\alpha$ and $\beta$ is a limit point of $C$, then $C \cap \beta \in \mathcal C_\beta$.  Furthermore, each such $\mathcal C_\alpha$ is nonempty, since the cofinality of $\alpha$ cannot change between $V$ and $M$.  For suppose $V \models \cf(\alpha) = \mu$ and $M \models \cf(\alpha) = \mu'<\kappa$.  Let $Y\in\p(\delta)^V$ code a witness to $V \models \cf(\alpha)=\mu$.  Then $Y \in M$, so $M \models \cf(\mu) = \mu'$.  We cannot have $\mu = \delta$ because $M \models \cf(\delta) = \kappa$, so $\mu<\delta$.  By elementarity, $M \models \cf(\mu) = \mu$.  Thus there is $X \in \p(\delta)^V$ that codes a prewellordering of $\delta$ of order-type $\alpha$ and a set $s \in \p_\kappa(\delta)^V$ such that $f_X[s]$ is club in $\alpha$.

Now suppose $V \models \cf(\alpha) = \kappa$.  Let $D \in V$ be a club in $\alpha$ of order-type $\kappa$.  Let $f : \delta \to \alpha$ be a surjection in $V$.  If $s$ is an initial segment of $D$ of limit order-type, then $r = f^{-1}[s] \in V$.  If $\beta = \sup(s)$, then $s = (f \downharpoonleft \beta) [r]$, so $s \in \mathcal C_\beta$.  Thus in $M$, there is a club $C \subseteq \alpha$ of order-type $\kappa$ such that all initial segments of limit length are in $\mathcal C_\beta$ for some $\beta<\alpha$.

Finally, suppose $V \models \cf(\alpha) = \delta$.  Let $D \in V$ be a club in $\alpha$ of order-type $\delta$, and let $\la \gamma_\beta : \beta < \delta \ra$ be its increasing enumeration.  Let $f : \delta \to \alpha$ be a surjection in $V$.  Let $g : \delta \to \delta$ be a function in $V$ such that for all $\beta<\delta$, $ f\circ g(\beta) = \gamma_\beta$.  In $M$, let $D' = \{ \gamma_\beta : \beta \in C^* \}$.  Let $s$ be an initial segment of $C^*$.  Let $\gamma = \sup(s)$, and let $\beta<\kappa$ be such that $s \subseteq z = \sigma_\gamma[\beta]$.  Then $g \restriction z$ is coded by an element of $\p_\kappa(\delta)^V$, and so $g \restriction s \in N$.  Thus $\{ \gamma_\beta : \beta \in s \} = f[r]$ for some $r \in N$.  In particular, there exists $C \in M$ that is club in $\alpha$, of order-type $\kappa$, and such that all initial segments of limit length are in $\mathcal C_\beta$ for some $\beta<\alpha$.
 
Although $M$ may not know which ordinals $\alpha$ of cofinality $\kappa$ have cofinality $\delta$ in $V$, we can just choose in either case some club $C \subseteq \alpha$ of order-type $\kappa$ such that all initial segments are in $\mathcal C_\beta$ for some $\beta<\alpha$.  Let $\mathcal C_\alpha = \{C\}$ for any such $C$.
This completes the construction of a $\square^*_\kappa$-sequence in $M$.  By elementarity, $V \models \square^*_\kappa$.
\end{proof}

Shelah \cite{MR675955} showed that if $I$ is a normal presaturated ideal on $\kappa^+$, then $I$ concentrates on $\{ \alpha : \cf(\alpha) = \cf(\kappa) \}$.  Thus the hypothesis of Theorem \ref{wps} can be simplified if we assume the ideal is presaturated.  On the other hand, a combination of theorems of Sargsyan \cite{MR2714009} and Woodin \cite{MR2723878} shows that we cannot drop the assumption that the ideal concentrates on the highest possible cofinality:

\begin{theorem}[Sargsyan-Woodin]
Assume the consistency of a Woodin limit of Woodin cardinals.  Then there is a model of ZFC satisfying:
\begin{enumerate}
\item Bounded Martin's Maximum, which implies $2^{\omega} = \omega_2$ and the tree property at $\omega_2$.
\item $\ns_{\omega_2} \upharpoonright \cof(\omega)$ is strong.
\end{enumerate} 
\end{theorem}

Theorem \ref{idealapprox} can be derived from Theorem \ref{wps}.  This is because for regular $\kappa$, if $I$ is normal ideal $I$ on $\kappa^+$ such that $\p(\kappa^+)/I$ has the $\kappa$-approximation property, then $I$ must concentrate on $\cof(\kappa)$.  This follows from Shelah's result that $\kappa^+ \cap \cof({<}\kappa)$ is approachable.  Indeed, any such ideal must contain the approachability ideal.  For suppose $S \in I^+$ and $\la a_\alpha : \alpha < \kappa^+ \ra$ witnesses that $S$ is approachable. Let $G \subseteq \p(\kappa^+)/I$ be generic with $S \in G$, and let $j : V \to M$ be the ultrapower embedding.  By \los' Theorem, there is $A \in M$, an unbounded subset of $(\kappa^+)^V$ of order-type $\leq \kappa$, such that all initial segments are in $\{ a_\alpha : \alpha < \kappa^+ \} \subseteq V$.  But this violates the $\kappa$-approximation property.

Foreman and Magidor \cite{MR1359154} asked whether a saturated normal ideal on $\omega_2$ can contain the approachability ideal.  This question appeared again in \cite{MR2768692}. Since $\square^*_\kappa$ implies that all subsets of $\kappa^+$ are approachable, Theorem \ref{wps} shows that the answer is ``no'' under the assumption that $2^\omega \leq \omega_2$.

If $I$ is a saturated ideal on $\omega_2$ and $2^\omega = \omega_1$, then forcing with $\p(\omega_2)/I$ does not add reals.  It is consistent relative to an almost-huge cardinal that there is a saturated ideal on $\omega_2$ whose associated Boolean algebra has a countably closed dense set, and in particular is a proper forcing  (see \cite{MR2768692}).  Remarkably, this is only possible under CH:

\begin{corollary}
\label{ch}
Suppose $I$ is a normal ideal on $\omega_2$.  Suppose either $I$ is weakly presaturated and $\p(\omega_2)/I$ is a proper forcing, or $I$ is presaturated and $\p(\omega_2)/I$ is a semiproper forcing.  Then the continuum hypothesis holds.
\end{corollary}

Before giving the proof, let us define the ``cofinal Strong Chang's Conjecture,'' abbreviated by $\mathrm{SCC}^{\cof}$ in \cite{MR3861321}.  This states that for every large enough cardinal $\theta$, every countable $M \prec H_\theta$, and every $\alpha<\omega_2$, there is a countable $N \prec H_\theta$ such that $M \subseteq N$, $M \cap \omega_1 = N \cap \omega_1$, and $\sup(N \cap \omega_2)>\alpha$.

\begin{proof}
Let $I$ be a normal ideal on $\omega_2$.  If $\p(\omega_2)/I$ is semiproper, then by Sakai \cite{MR2191239}, $\mathrm{SCC}^{\cof}$ holds.  By Todor\v{c}evi\'{c} \cite{MR1261218}, $\mathrm{SCC}^{\cof}$ implies $2^\omega \leq \omega_2$.  By Torres-Perez and Wu \cite{MR3431031}, if $\mathrm{SCC}^{\cof}$ holds, then the failure of CH is equivalent to the tree property at $\omega_2$.  If $\p(\omega_2)/I$ is a proper forcing, then it cannot change the cofinality $\omega_2$ to $\omega$, so $I$ must concentrate on $\cof(\omega_1)$.  If $I$ is presaturated, then it concentrates on $\cof(\omega_1)$ by Shelah's Theorem.  In either case, the hypotheses imply that $2^\omega \leq \omega_2$ and there is a weakly presaturated ideal concentrating on $\cof(\omega_1)$, which by Theorem \ref{wps} implies that the tree property at $\omega_2$ fails.  Therefore, $2^\omega = \omega_1$.
\end{proof}

\section{Weak square from lifted embeddings}
\label{wsli}

The known methods for forcing either $(\mu^{++},\mu^+)\chang(\mu^+,\mu)$ or the existence of a saturated ideal on $\mu^+$, where $\mu$ is uncountable, start with a huge or almost-huge cardinal $\kappa>\mu$ with witnessing embedding $j : V \to M$, and collapse $\kappa$ to $\mu^+$ and $j(\kappa)$ to $\mu^{++}$ in a way that allows the embedding to be generically lifted.  A variety of such constructions are described in \cite{MR2768692}.  These constructions typically force $\mu^{<\mu}=\mu$, and thus $\square^*_\mu$.  However, we argue here that $\square^*_\mu$ is already guaranteed by certain abstract features of these forcings, without any \emph{prima facie} assumptions about the effect on cardinal arithmetic.  The arguments will apply to embeddings coming from hypotheses weaker than almost-hugeness.

In typical situations, we lift an almost-huge embedding $j : V \to M$ through a forcing $\mathbb P * \dot{\mathbb Q}$, obtaining $j' : V[G*H] \to M[G'*H']$, where $\mathbb P$ is $\crit(j)$-c.c.\ and $\dot{\mathbb Q}$ is not.  In order to lift through $H$, it is usually key to the argument that at least small pieces of $G*H$ are members of $M[G']$.  The next observation shows that this is enough to guarantee that hypothesis (\ref{jtyp_lift}) of Theorem \ref{typcol} holds.

\begin{proposition}
Suppose $\kappa\leq\delta$ are regular cardinals, $M \subseteq V$ is a ${<}\delta$-closed inner model, and $\mathbb P$ is a $\delta$-c.c.\ partial order.  Let $G \subseteq \mathbb P$ be generic over $V$, and suppose $N$ is an outer model of $M$.  Then the following are equivalent:
\begin{enumerate}
\item \label{pko} $\p_\kappa(\ord)^{V[G]} \subseteq N$.
\item \label{pkp} $\p_\kappa(\mathbb P)^{V[G]} \subseteq N$.
\end{enumerate}
\end{proposition}

\begin{proof}
We may assume that the elements of $\mathbb P$ are ordinals.  Then (\ref{pko}) $\Rightarrow$ (\ref{pkp}) is trivial.  For the other direction, suppose $\xi<\kappa$ and $f : \xi \to \ord$ is in $V[G]$.  Let $\dot f$ be a name for $f$, and let $p_0 \Vdash \dom \dot f = \check \xi$.  For $\alpha<\xi$, let $A_\alpha$ be a maximal antichain below $p_0$ deciding $\dot f(\check\alpha)$.  For $\alpha<\xi$ and $p\in A_\alpha$, let $\beta_{\alpha,p}$ be the value of $\dot f(\check\alpha)$ decided by $p$.  Define a $\mathbb P$-name:
	$$\tau := \{ (p,\la\check\alpha,\check\beta_{\alpha,p}\ra) : \alpha<\xi \text{ and } p \in A_\alpha \}.$$
	Then $p_0 \Vdash \dot f = \tau$.  By the $\delta$-c.c.\ and the ${<}\delta$-closure of $M$, $\tau \in M$.  Now for all $\alpha<\xi$ there is a unique $q_\alpha \in G\cap A_\alpha$, and $\tau^G$ can be computed from $\tau$ and $\{ q_\alpha : \alpha<\xi \}$.  By hypothesis, $\{ q_\alpha : \alpha<\xi \} \in N$, and therefore $f \in N$.
\end{proof}

\begin{theorem}
\label{typcol}
Suppose $j : V \to M$ is an elementary embedding with critical point $\kappa$ definable from parameters in $V$.  Suppose $\mathbb P* \dot{\mathbb Q}$ is a two-step iteration such that:
\begin{enumerate}
\item\label{jtyp_size} $M$ is $|\mathbb P|$-closed, and $|\mathbb P| < j(\kappa)$.% \in M$ and $M \models |\mathbb P|<j(\kappa)$.
\item\label{jtyp_col} $\mathbb P* \dot{\mathbb Q}$ collapses all ordinals in the open interval $(\kappa,j(\kappa))$.
%\item\label{jtyp_closure} Whenever $G \subseteq \mathbb P$ is generic over $V$, then $\p_\kappa(\ord)^{V[G]} = \p_\kappa(\ord)^{M[G]}$.
\item \label{jtyp_lift} Whenever $G*H$ is $\mathbb P* \dot{\mathbb Q}$-generic over $V$, then in some outer model, $j$ can be lifted to $j' : V[G*H] \to M[G' * H']$, such that $\p_\kappa(\ord)^{V[G*H]} \subseteq M[G']$.
\end{enumerate}
Then $\mathbb P$ forces that $\kappa= \mu^+$ for some $\mu<\kappa$, and $\square^*_\mu$ holds.
\end{theorem}

\begin{proof}
First we claim that $\dot{\mathbb Q}$ is forced to be $\kappa$-distributive.  Let $G*H \subseteq \mathbb P *\dot{\mathbb Q}$ be generic.  Suppose $r \in \p_\kappa(\ord) \cap V[G*H] \setminus V[G]$.  Let $j' : V[G*H] \to M[G' * H']$ be a lifting of $j$ as in (\ref{jtyp_lift}).  Note that $j'(r) = j[r]$, and there is an $s \in \p_\kappa(\ord)^{V[G*H]}$ that codes $j \restriction r$.  Since $s \in M[G']$, we have that $j'(r) \in M[G']$. 
But by elementarity, $j'(r) \in M[G' * H'] \setminus M[G']$, a contradiction.

Next we claim that $\p(\kappa)^{V[G*H]} \subseteq M[G']$.  Suppose $X \in\p(\kappa)^{V[G*H]}$.  Then $X = j'(X) \cap \kappa$ is in $M[G'*H']$.  But since $M[G'] \models$ ``$j(\dot{\mathbb Q})^{G'}$ is $j(\kappa)$-distributive,'' $X \in M[G']$.  Note that since $|\mathbb P|$ is collapsed to $\kappa$, $G$ is coded by a subset of $\kappa$ in $V[G*H]$, and thus $G \in M[G']$.

By hypothesis (\ref{jtyp_size}), $\p_\kappa(\ord)^{V[G]} = \p_\kappa(\ord)^{M[G]}$.
In $M[G']$, define a sequence $\la \mathcal C_\alpha : \alpha < j(\kappa) \ra$ as follows.  If $M[G] \models \cf(\alpha)<\kappa$, let $\mathcal C_\alpha$ be the set of all clubs in $\alpha$ of order-type $<\kappa$ that live in $M[G]$.  Since 
$|\mathbb P| < j(\kappa)$ 
and $j(\kappa)$ is inaccessible in $M$, each such set has size $<j(\kappa)$ in $M[G]$, and thus size $\leq\kappa$ in $M[G']$.  If $M[G] \models \cf(\alpha)\geq\kappa$, then this is true in $V[G]$, and it remains true in $V[G*H]$ by distributivity.  In $V[G*H]$, $|\alpha| = \kappa$, so there is a club $C \subseteq \alpha$ of order-type $\kappa$.  All its initial segments are in $\p_\kappa(\ord)^{M[G]}$.   $C$ is coded by some $X \subseteq \kappa$, which is in $M[G']$.  Working in $M[G']$ we let $\mathcal C_\alpha = \{ D \}$, where $D$ is any club in $\alpha$ of order-type $\kappa$ such that all initial segments are in $M[G]$. 

Since $j(\kappa)$ is not a limit cardinal in $M[G']$, it follows by elementarity that $\kappa = \mu^+$ in $V[G]$ for some $\mu<\kappa$, and thus $j(\kappa)= \mu^+$ in $M[G']$.  We conclude using Lemma \ref{otreduct} that $\square^*_\mu$ holds in $M[G']$.  By elementarity, $\square^*_\mu$ holds in $V[G]$.
\end{proof}

If we weaken hypothesis (\ref{jtyp_col}) of Theorem \ref{typcol} to say just that all ordinals in some final segment of $j(\kappa)$ are collapsed, then the argument does not go through:

\begin{proposition}
Suppose $\kappa$ is measurable, and $j : V \to M$ is derived from a normal measure on $\kappa$.  Then there is a two-step iteration $\mathbb P * \dot{\mathbb Q}$ such that:
\begin{enumerate}
\setcounter{enumi}{-1}
\item\label{ext_tp} $\mathbb P * \dot{\mathbb Q}$ forces that $\kappa=\omega_2$ and $\omega_2$ has the tree property.
\item \label{ext_size} $|\mathbb P| = \kappa$. % \in M$ and $M \models |\mathbb P|<j(\kappa)$.
\item\label{ext_col} $\mathbb P* \dot{\mathbb Q}$ collapses all ordinals in the open interval $(\kappa^+,j(\kappa))$.
\item\label{ext_absorb} Whenever $G*H$ is $\mathbb P* \dot{\mathbb Q}$-generic over $V$, then in some outer model, $j$ can be lifted to 
$j' : V[G*H] \to M[G' * H']$, such that $\p_{\kappa^+}(\ord)^{V[G*H]} \subseteq M[G']$.
\end{enumerate}
\end{proposition}

\begin{proof}
Let $j : V \to M$ be as hypothesized, and note that $M^\kappa \subseteq M$.  Let $\mathbb P$ be any $\kappa$-c.c.\ forcing of size $\kappa$ that forces $\kappa = \omega_2$ and the tree property holds at $\omega_2$, such as Mitchell's forcing.  In $V^{\mathbb P}$, let $\mathbb Q = \col(\kappa^+,{<}j(\kappa))$.  Let $G * H \subseteq \mathbb P * \dot{\mathbb Q}$ be generic.  $\mathbb Q$ preserves the tree property since it does not add subsets of $\kappa$.  Thus (\ref{ext_tp}), (\ref{ext_size}), and (\ref{ext_col}) hold.  

Since $\mathbb P$ is $\kappa$-c.c.\ and $\kappa = \crit(j)$, $\mathbb P$ is a regular suborder of $j(\mathbb P)$, and thus a further forcing yields $G' \subseteq j(\mathbb P)$ such that $G = G' \cap V_\kappa$.  Thus we can extend the embedding to $j : V[G] \to M[G']$.  Furthermore, $j[H]$ generates a filter $H'$ that is $j(\mathbb Q)$-generic over $M[G']$.  This is because for each dense open $D \subseteq j(\mathbb Q)$ in $M[G']$, there is a function $f : \kappa \to \p(\mathbb Q)$ in $V[G]$ such that $f(\alpha)$ is a dense open subset of $\mathbb Q$ for all $\alpha< \kappa$, and $D = j(f)(\kappa)$.  If $E = \bigcap_{\alpha<\kappa} f(\alpha)$, then $E$ is a dense open subset of $\mathbb Q$, and $j(E) \subseteq D$.  Since $H$ is generic, there is $q \in E \cap H$, so $j(q) \in D \cap H'$.  Since $\mathbb Q$ is $\kappa^+$-closed, $\p_{\kappa^+}(\ord)^{V[G*H]} = \p_{\kappa^+}(\ord)^{V[G]} = \p_{\kappa^+}(\ord)^{M[G]} \subseteq M[G']$, establishing (\ref{ext_absorb}).
\end{proof}

\begin{remark}
The above argument also applies to embeddings derived from short extenders.
\end{remark}

Nonetheless, we can carry out a similar argument as for Theorem \ref{typcol} under weaker collapsing conditions, by adding more assumptions about the forcing.  The idea is that $\square^*_\mu$ will be forced whenever $\mathbb P * \dot{\mathbb Q}$ forces that $j(\kappa)$ is the successor of a cardinal $\lambda$ satisfying $\lambda^{<\lambda} = \lambda$, and $M^{j(\mathbb P)}$ can see enough of this structure.

\begin{proposition}
Suppose that $j : V \to M$ is an elementary embedding with critical point $\kappa$ definable from parameters in $V$, and $\mu,\lambda$ are regular such that $\mu<\kappa\leq \lambda < j(\kappa)$.  Suppose $\mathbb P* \dot{\mathbb Q}$ is such that:
\begin{enumerate}
\item $\mathbb P * \dot{\mathbb Q} \in M$.
\item\label{jtyp_sizes1} $M \models |\mathbb P|<j(\kappa)$, $\mathbb P * \dot{\mathbb Q}$ is $j(\kappa)$-c.c., and $\mathbb P * \dot{\mathbb Q} \subseteq V_{j(\kappa)}$.
\item $\mathbb P * \dot{\mathbb Q}$ forces over $M$ that $\kappa = \mu^+$ and $j(\kappa) = \lambda^{+}$.
\item\label{jtyp_closure1} $\Vdash^M_{\mathbb P} \dot{\mathbb Q}$ is $\lambda$-distributive.
\item \label{jtyp_lift1} Whenever $G*H$ is $\mathbb P* \dot{\mathbb Q}$-generic over $V$, then in some outer model, $j$ can be lifted to $j' : V[G*H] \to M[G' * H']$, such that for all $\alpha<j(\kappa)$, $G*H_\alpha \in M[G']$, where $H_\alpha = H \cap V_\alpha$.
\end{enumerate}
Then $\mathbb P$ forces $\square^*_\mu$.
\end{proposition}

\begin{proof}
Let $G*H \subseteq \mathbb P *\dot{\mathbb Q}$ be generic, and let $j' : V[G*H] \to M[G' * H']$ be a lifting of $j$ as in (\ref{jtyp_lift1}).
In $M[G']$, we can define the set $A \subseteq j(\kappa)$ of ordinals that have cofinality $<\lambda$ in $M[G]$, which is the same as the set of ordinals $<j(\kappa)$ that have cofinality $<\lambda$ in $M[G*H]$ by $\lambda$-distributivity.

In $M[G']$, define a sequence $\la \mathcal C_\alpha : \alpha < j(\kappa) \ra$ as follows.  If $\alpha \in A$, let $\mathcal C_\alpha$ be the set of all clubs in $\alpha$ of order-type $<\lambda$ that live in $M[G]$.  Since $|\mathbb P|<j(\kappa)$, this set has size $<j(\kappa)$ in $M[G']$.  If $\alpha \in j(\kappa)\setminus A$, let $\mathcal C_\alpha = \{ D\}$, where $D$ is any club in $\alpha$ of order-type $\lambda$ such that all initial segments are in $M[G]$.  Such a club exists in $M[G']$ because there is one in $M[G*H_\alpha]$ for some $\alpha<j(\kappa)$.  Conclude using Lemma \ref{otreduct} that $\square^*_\mu$ holds in $M[G']$ and thus in $V[G]$ by elementarity.
\end{proof}

We finish by giving an example to clear up a possible misconception.  In the situation of Proposition \ref{changsquare1} and Theorem \ref{idealapprox}, we derive $\square^*_\kappa$ from the fact that $\square^*_{\kappa^+}$ holds in $V$, a generic ultrapower $M$ can see enough information about $V$ to know this, and the ultrapower embedding allows us to reflect this downward.  But this does not characterize all of the situations we have discussed.  In fact, we can force a saturated ideal on $\omega_2$ along with the tree property at $\omega_3$ using conventional methods.  The weak square sequence of length $j(\kappa)$ as constructed in Theorem \ref{typcol} may only exist in a generic extension that we do not actually wish to take, but its virtual existence is enough to ensure the failure of the tree property at $\kappa$ in the universe of interest.

\begin{proposition}
If there is a huge cardinal, then there is a generic extension in which there is a saturated ideal on $\omega_2$ and the tree property holds at $\omega_3$.
\end{proposition}

\begin{proof}
Suppose $\kappa$ is huge.  In particular, there is an almost-hugeness embedding $j : V \to M$ with critical point $\kappa$ such that $\delta=j(\kappa)$ is weakly compact.  By Magidor's modification of Kunen's construction (see \cite{MR2768692}), there is a countably closed forcing $\mathbb P \subseteq V_\kappa$ that turns $\kappa$ into $\omega_2$ and is such that $j(\mathbb P)$ is $\delta$-c.c.\ and $j(\mathbb P)$ projects to $\mathbb P * \dot\col(\kappa,{<}\delta)$.  Furthermore, whenever $G * H \subseteq \mathbb P * \dot\col(\kappa,{<}\delta)$ is generic, then the quotient $j(\mathbb P)/(G *H)$ forces that there is a lifting of $j$ to $j' : V[G*H] \to M[G' * H']$.  

Next, let $\nu$ be either $\omega$ or $\omega_1$.  By GCH, $\add(\nu,\delta)$ is $\kappa$-c.c.  If $K \subseteq \add(\nu,\delta)$ is generic over $V[G*H]$, then further forcing allows us to lift the embedding to $j'' : V[G*H*K] \to M[G'*H'*K']$, where $K'$ is generic over $V[G']$ for $\add(\nu,j(\delta)\setminus j[\delta])$.  Note that this is $\delta$-c.c.\ in $V[G']$.

Now it is well-known that Mitchell's forcing $\mathbb M$ for the tree property at $\delta$ is a projection of $\add(\nu,\delta) \times \col(\kappa,{<}\delta)$.  Let $Q$ be the projection of $K \times H$, which is an $\mathbb M$-generic filter over $V[G]$.  There is a $\delta$-c.c.\ forcing $\mathbb R$ in $V[G*Q]$ that produces the lifted embedding $j''$.  In $V[G*Q]$, we define a normal ideal $I$ on $\omega_2$ by
$$I = \{ X \subseteq \kappa : 1 \Vdash_{\mathbb R} \kappa \notin j''(X) \}.$$
If $X_0,X_1 \in I^+$, then there are $r_0,r_1 \in \mathbb R$ such that $r_i \Vdash \kappa \in j''(X_i)$.  If $X_0 \cap X_1 \in I$, then $r_0,r_1$ are incompatible.  Thus $\p(\kappa)/I$ is $\delta$-c.c.
In summary, $V[G*Q]$ has a saturated ideal on $\kappa = \omega_2$ and satisfies the tree property at $\delta=\omega_3$.
\end{proof}

\begin{remark}
Using a suitable modification of Mitchell's forcing, we can similarly obtain a model of $(\omega_3,\omega_2)\chang(\omega_2,\omega_1)$ plus the tree property at $\omega_3$, starting from a huge cardinal.
\end{remark}

\bibliographystyle{amsplain.bst}
\bibliography{compactvhuge.bib}

\providecommand{\bysame}{\leavevmode\hbox to3em{\hrulefill}\thinspace}
\providecommand{\MR}{\relax\ifhmode\unskip\space\fi MR }
\providecommand{\MRhref}[2]{%
  \href{http://www.ams.org/mathscinet-getitem?mr=#1}{#2}
}
\providecommand{\href}[2]{#2}
\begin{thebibliography}{10}

\bibitem{MR401472}
James~E. Baumgartner, \emph{Almost-disjoint sets, the dense set problem and the
  partition calculus}, Ann. Math. Logic \textbf{9} (1976), no.~4, 401--439.
  \MR{401472}

\bibitem{MR556894}
James~E. Baumgartner and Richard Laver, \emph{Iterated perfect-set forcing},
  Ann. Math. Logic \textbf{17} (1979), no.~3, 271--288. \MR{556894}

\bibitem{MR2963017}
Benjamin Claverie and Ralf Schindler, \emph{Woodin's axiom {$(\ast)$}, bounded
  forcing axioms, and precipitous ideals on {$\omega_1$}}, J. Symbolic Logic
  \textbf{77} (2012), no.~2, 475--498. \MR{2963017}

\bibitem{MR3861321}
Sean~D. Cox, \emph{Chang's conjecture and semiproperness of nonreasonable
  posets}, Monatsh. Math. \textbf{187} (2018), no.~4, 617--633. \MR{3861321}

\bibitem{MR645907}
D.~Donder, R.~B. Jensen, and B.~J. Koppelberg, \emph{Some applications of the
  core model}, Set theory and model theory ({B}onn, 1979), Lecture Notes in
  Math., vol. 872, Springer, Berlin-New York, 1981, pp.~55--97. \MR{645907}

\bibitem{MR730856}
Hans-Dieter Donder and Peter Koepke, \emph{On the consistency strength of
  ``accessible'' {J}\'{o}nsson cardinals and of the weak {C}hang conjecture},
  Ann. Pure Appl. Logic \textbf{25} (1983), no.~3, 233--261. \MR{730856}

\bibitem{2018arXiv181211768E}
Monroe {Eskew} and Yair {Hayut}, \emph{{Global Chang's Conjecture and singular
  cardinals}}, arXiv e-prints (2018), arXiv:1812.11768.

\bibitem{MR2768692}
Matthew Foreman, \emph{Ideals and generic elementary embeddings}, Handbook of
  set theory. {V}ols. 1, 2, 3, Springer, Dordrecht, 2010, pp.~885--1147.
  \MR{2768692}

\bibitem{MR1359154}
Matthew Foreman and Menachem Magidor, \emph{Large cardinals and definable
  counterexamples to the continuum hypothesis}, Ann. Pure Appl. Logic
  \textbf{76} (1995), no.~1, 47--97. \MR{1359154}

\bibitem{MR2063629}
Joel~David Hamkins, \emph{Extensions with the approximation and cover
  properties have no new large cardinals}, Fund. Math. \textbf{180} (2003),
  no.~3, 257--277. \MR{2063629}

\bibitem{MR309729}
R.~Bj\"{o}rn Jensen, \emph{The fine structure of the constructible hierarchy},
  Ann. Math. Logic \textbf{4} (1972), 229--308; erratum, ibid. 4 (1972), 443,
  With a section by Jack Silver. \MR{309729}

\bibitem{MR520190}
A.~Kanamori and M.~Magidor, \emph{The evolution of large cardinal axioms in set
  theory}, Higher set theory ({P}roc. {C}onf., {M}ath. {F}orschungsinst.,
  {O}berwolfach, 1977), Lecture Notes in Math., vol. 669, Springer, Berlin,
  1978, pp.~99--275. \MR{520190}

\bibitem{MR3647840}
John Krueger, \emph{Forcing with adequate sets of models as side conditions},
  MLQ Math. Log. Q. \textbf{63} (2017), no.~1-2, 124--149. \MR{3647840}

\bibitem{MR313057}
William Mitchell, \emph{Aronszajn trees and the independence of the transfer
  property}, Ann. Math. Logic \textbf{5} (1972/73), 21--46. \MR{313057}

\bibitem{MR3201836}
Itay Neeman, \emph{Forcing with sequences of models of two types}, Notre Dame
  J. Form. Log. \textbf{55} (2014), no.~2, 265--298. \MR{3201836}

\bibitem{MR2191239}
Hiroshi Sakai, \emph{Semiproper ideals}, Fund. Math. \textbf{186} (2005),
  no.~3, 251--267. \MR{2191239}

\bibitem{MR2714009}
Grigor Sargsyan, \emph{A tale of hybrid mice}, ProQuest LLC, Ann Arbor, MI,
  2009, Thesis (Ph.D.)--University of California, Berkeley. \MR{2714009}

\bibitem{MR675955}
Saharon Shelah, \emph{Proper forcing}, Lecture Notes in Mathematics, vol. 940,
  Springer-Verlag, Berlin-New York, 1982. \MR{675955}

\bibitem{MR1318912}
\bysame, \emph{Cardinal arithmetic}, Oxford Logic Guides, vol.~29, The
  Clarendon Press, Oxford University Press, New York, 1994, Oxford Science
  Publications. \MR{1318912}

\bibitem{MR1261218}
Stevo Todor\v{c}evi\'{c}, \emph{Conjectures of {R}ado and {C}hang and cardinal
  arithmetic}, Finite and infinite combinatorics in sets and logic ({B}anff,
  {AB}, 1991), NATO Adv. Sci. Inst. Ser. C Math. Phys. Sci., vol. 411, Kluwer
  Acad. Publ., Dordrecht, 1993, pp.~385--398. \MR{1261218}

\bibitem{MR3431031}
V\'{\i}ctor Torres-P\'{e}rez and Liuzhen Wu, \emph{Strong {C}hang's conjecture
  and the tree property at {$\omega_2$}}, Topology Appl. \textbf{196} (2015),
  no.~part B, 999--1004. \MR{3431031}

\bibitem{MR2723878}
W.~Hugh Woodin, \emph{The axiom of determinacy, forcing axioms, and the
  nonstationary ideal}, revised ed., De Gruyter Series in Logic and its
  Applications, vol.~1, Walter de Gruyter GmbH \& Co. KG, Berlin, 2010.
  \MR{2723878}

\end{thebibliography}

\end{document}